\documentclass[11pt,letterpaper]{article} 

\usepackage{amsmath,amsfonts,amsthm,amssymb,stmaryrd,relsize}
\theoremstyle{plain}

\numberwithin{equation}{section}

\newtheorem{thm}{Theorem}[section]
\newtheorem{lem}[thm]{Lemma}
\newtheorem{cor}[thm]{Corollary}

{\begin{flushleft}\textbf{Example #1}.\enspace}%
{\end{flushleft}}

\allowdisplaybreaks  

\newcommand{\rmtr}{\mathrm{tr\,}}
\newcommand{\rmhs}{\mathrm{HS}}

\newcommand{\ascript}{\mathcal{A}}
\newcommand{\bscript}{\mathcal{B}}
\newcommand{\lscript}{\mathcal{L}}

\newcommand{\ab}[1]{\left|#1\right|}
\newcommand{\doubleab}[1]{\left|\left|#1\right|\right|}
\newcommand{\brac}[1]{\left\{#1\right\}}
\newcommand{\paren}[1]{\left(#1\right)}
\newcommand{\sqbrac}[1]{\left[#1\right]}
\newcommand{\elbows}[1]{{\left\langle#1\right\rangle}}
\newcommand{\ket}[1]{{\left|#1\right>}}
\newcommand{\bra}[1]{{\left<#1\right|}}

\begin{document}

\title{OPERATOR ISOMORPHISMS ON\\HILBERT SPACE TENSOR PRODUCTS}
\author{Stan Gudder\\ Department of Mathematics\\
University of Denver\\ Denver, Colorado 80208\\
sgudder@du.edu}
\date{}
\maketitle

\begin{abstract}
This article presents an isomorphism between two operator algebras $L_1$ and $L_2$ where $L_1$ is the set of operators on a space of
Hilbert-Schmidt operators and $L_2$ is the set of operators on a tensor product space. We next compare our isomorphism to a well-known result called Choi's isomorphism theorem. The advantage of Choi's isomorphism is that it takes completely positive maps to positive operators. One advantage of our isomorphism is that it applies to infinite dimensional Hilbert spaces, while Choi's isomorphism only holds for finite dimensions. Also, our isomorphism preserves operator products while Choi's does not. We close with a brief discussion on some uses of our isomorphism.
\end{abstract}

\section{Operator Isomorphisms}  
For separable complex Hilbert spaces $H_1$, $H_2$, we denote the set of bounded linear operators from $H_1$ to $H_2$ by
$\lscript (H_1,H_2)$. If $H_1=H_2$ we write $\lscript (H)=\lscript (H,H)$. If $A\in\lscript (H_1,H_2)$, we define $A^*\in\lscript (H_2,H_1)$ as the unique operator satisfying $\elbows{A\phi ,\psi}=\elbows{\phi ,A^*\psi}$ for all $\phi\in H_1$, $\psi\in H_2$. We say that $A\in\lscript (H_1,H_2)$ is a \textit{Hilbert-Schmidt} (HS) \textit{operator} if $\rmtr (A^*A)<\infty$. Notice that $A^*A\in\lscript (H_1)$. Since $\rmtr (A^*A)=\rmtr (AA^*)$, it follows that if $A$ is HS, then $A^*$ is HS. We denote the set of HS operators from $H_1$ to $H_2$ by $\lscript _{\rmhs}(H_1,H_2)$. Clearly
$\lscript _{\rmhs}(H_1,H_2)$ is a complex linear space and it is well-known that $AB\in\lscript _{\rmhs}(H_1,H_2)$ for all $A\in\lscript (H_2)$,
$B\in\lscript _{\rmhs}(H_1,H_2)$. If we define the inner product $\elbows{A,B}_{\rmhs}=\rmtr (A^*B)$, then it is known that
$\lscript _{\rmhs}(H_1,H_2)$ becomes a Hilbert space. The following theorem concerning $\lscript _{\rmhs}(H_1,H_2)$ is well-known \cite{nc00}.

\begin{thm}    
\label{thm11}
Let $\brac{\phi _i}$ and $\brac{\psi _i}$ be orthonormal bases for $H_1$ and $H_2$, respectively.
{\rm{(a)}}\enspace $\brac{\ket{\psi _i}\bra{\phi _j}}$ is an orthonormal basis for $\lscript _{\rmhs}(H_1,H_2)$.
{\rm{(b)}}\enspace If $A\in\lscript (H_1,H_2)$, then the following statements are equivalent:\newline
{\rm{(i)}}\enspace $A\in\lscript _{\rmhs}(H_1,H_2)$,
{\rm{(ii)}}\enspace $\sum\ab{\elbows{\psi _i,A\phi _j}}^2<\infty$,
{\rm{(iii)}}\enspace $A=\sum\elbows{\psi _i,A\phi _j}\ket{\psi _i}\bra{\phi _j}$.
\end{thm}

The main concern of this article is a $C^*$-algebra isomorphism between $\lscript\paren{\lscript _{\rmhs}(H_1,H_2)}$ and
$\lscript (H_1\otimes H_2)$. We first define a Hilbert space isomorphism from $\lscript _{\rmhs}(H_1,H_2)$ to $H_1\otimes H_2$. Let
$\brac{\phi _i}$ and $\brac{\psi _i}$ be orthonormal bases for $H_1$ and $H_2$, respectively, and let
$\ascript =\brac{\phi _i,\psi _j\colon i,j=1,2,\ldots}$. Define the map $J_\ascript\colon\lscript _{\rmhs}(H_1,H_2)\to H_1\otimes H_2$ given by
\begin{equation}                
\label{eq11}
J_\ascript (A)=\sum\elbows{\psi _i,A\phi _j}\phi _j\otimes\psi _i
\end{equation}
Notice that $J_\ascript\paren{\ket{\psi _i}\bra{\phi _j}}=\phi _j\otimes\psi _i$.

\begin{thm}    
\label{thm12}
The map $J_\ascript$ is a Hilbert space isomorphism, $J_\ascript ^*\colon H_1\otimes H_2\to\lscript _{\rmhs}(H_1,H_2)$ satisfies
$J_\ascript ^*=J_\ascript ^{-1}$ and
\begin{equation}                
\label{eq12}
J_\ascript ^*(\alpha )=\sum\elbows{\phi _j\otimes\psi _i,\alpha}\ket{\psi _i}\bra{\phi _j}
\end{equation}
\end{thm}
\begin{proof}
It is clear that $J_\ascript$ is linear. Since $J_\ascript$ takes the basis $\brac{\ket{\psi _i}\bra{\phi _j}}$ to the basis $\brac{\phi _j\otimes\psi _i}$, it follows that $J_\ascript$ preserves inner products. We then conclude that $J_\ascript$ is injective. To show that $J_\ascript$ is surjective, let
$\alpha\in H_1\otimes H_2$. Letting $A\in\lscript _{\rmhs}(H_1,H_2)$ be defined by
$A=\sum\elbows{\phi _i\otimes\psi _j,\alpha}\ket{\psi _j}\bra{\phi _i}$ we have that
\begin{align*}
J_\ascript (A)&=\sum _{i,j}\bra{\psi _i}\sum _{r,s}\elbows{\phi _r\otimes\psi _s,\alpha}\ket{\psi _s}\elbows{\phi _r,\phi _j}\phi _j\otimes\psi _i\\
   &=\sum _{r,s}\elbows{\phi _r\otimes\psi _s,\alpha}\sum _{i,j}\delta _{i,s}\delta _{r,j}\phi _j\otimes\psi _i\\
   &=\sum _{i,j}\elbows{\phi _j\otimes\psi _i,\alpha}\phi _j\otimes\psi _i=\alpha
\end{align*}
so $J_\ascript$ is surjective. Since $J_\ascript$ preserves inner products, we have 
\begin{equation*}
\elbows{J_\ascript ^*J_\ascript (A),B}_{\rmhs}=\elbows{J_\ascript (A),J_\ascript (B)}=\elbows{A,B}_{\rmhs}
\end{equation*}
Hence, $J_\ascript ^*J_\ascript =I$ and similarly $J_\ascript J_\ascript ^*=I$ so $J_\ascript ^*=J_\ascript ^{-1}$. To prove \eqref{eq12} we have for $A\in\lscript _{\rmhs}(H_1,H_2)$, $\alpha\in H_1\otimes H_2$ that
\begin{align*}
\elbows{A,\sum\elbows{\phi _j\otimes\psi _i,\alpha}\ket{\psi _i}\bra{\phi _j}}_{\rmhs}
  &=\sum\elbows{\phi _j\otimes\psi _i,\alpha}\elbows{A,\ket{\psi _i}\bra{\phi _j}}_{\rmhs}\\
  &=\sum\elbows{\phi _j\otimes\psi _i,\alpha}\rmtr\paren{A^*\ket{\psi _i}\bra{\phi _j}}\\
  &=\sum\elbows{\phi _j\otimes\psi _i,\alpha}\elbows{\phi _jA^*\psi _i}\\
  &=\sum\elbows{A\phi _j,\psi _i}\elbows{\phi _j\otimes\psi _i,\alpha}\\
  &=\elbows{\sum\elbows{\psi _i,A\phi _j}\phi _j\otimes\psi _i,\alpha}\\
  &=\elbows{J_\ascript (A),\alpha}=\elbows{A,J_\ascript ^*(\alpha )}
\end{align*}
Hence, \eqref{eq12} holds.
\end{proof}

It follows from Theorem~\ref{thm12} that $J_\ascript ^*(\phi _i\otimes\psi _j)=\ket{\psi _j}\bra{\phi _i}$.

An anti-linear map $K\colon H\to H$ is called a \textit{conjugation} if there exists an orthonormal basis $\bscript =\brac{\phi _i}$ for $H$ such that $K\phi _i=\phi _i$, $i=1,2,\ldots\,$. Thus, $K$ is a conjugation if and only if its spectrum $\sigma (K)=\brac{1}$. Since $K$ depends on the basis
$\bscript$, we sometimes write $K=K_\bscript$.

\begin{lem}    
\label{lem13}
$K\colon H\to H$ is a conjugation if and only if there exists an orthonormal basis $\brac{\phi _i}$ for $H$ such that
\begin{equation}                
\label{eq13}
K\phi =\sum\elbows{\phi ,\phi _i}\phi _i
\end{equation}
for all $\phi\in H$.
\end{lem}
\begin{proof}
If \eqref{eq13} holds, then $K\phi _j=\sum\elbows{\phi _j,\phi _i}\phi _i=\phi _j$, $i=1,2,\ldots\,$. Conversely, suppose there exists an orthonormal basis $\brac{\phi _i}$ with $K\phi _i=\phi _i$, $i=1,2,\ldots\,$. Assuming that $K$ is anti-linear, we have for all $\phi\in H$ that
\begin{equation*}
K\phi =K\paren{\sum\elbows{\phi _i,\phi}\phi _i}=\sum\elbows{\phi ,\phi _i}K\phi _i=\sum\elbows{\phi ,\phi _i}\phi _i\qedhere
\end{equation*}
\end{proof}

\begin{cor}    
\label{cor14}
If $K$ is a conjugation, then 
{\rm{(a)}}\enspace $K^2=I$,
{\rm{(b)}}\enspace $K^*=K$,
{\rm{(c)}}\enspace $\elbows{K\phi ,K\psi}=\elbows{\psi ,\phi}$.
\end{cor}
\begin{proof}
(a)\enspace Since $K\phi _i=\phi _i$ for an orthonormal basis $\brac{\phi _i}$, we have
\begin{equation*}
K^2\phi _i=K\phi _i=\phi _i
\end{equation*}
Since $K$ is linear, $K^2=I$.
(b)\enspace To show that $K^*=K$, applying \eqref{eq13} we have for all $\phi ,\psi\in H$ that
\begin{align*}
\elbows{K^*\phi ,\psi}&=\elbows{K\psi ,\phi}=\elbows{\sum\elbows{\psi ,\phi _i}\phi _i,\phi}=\sum\elbows{\phi _i,\psi}\elbows{\phi _i,\phi}\\
   &=\elbows{\sum\elbows{\phi ,\phi _i}\phi _i,\psi}=\elbows{K\phi ,\psi}
\end{align*}
Hence, $K^*\phi =K\phi$ so $K^*=K$.
(c)\enspace By (a) and (b) we have that
\begin{equation*}
\elbows{K\phi ,K\psi}=\elbows{K^*K\psi ,\phi}=\elbows{K^2\psi ,\phi}=\elbows{\psi ,\phi}\qedhere
\end{equation*}
\end{proof}

It follows from Corollary~\ref{cor14}(c) that a conjugation takes an orthonormal basis to an orthonormal basis.

\begin{lem}    
\label{lem15}
Letting $\ascript =\brac{\phi _i,\psi _j}$, $\bscript =\brac{\phi _i}$ we have that
\begin{equation*}
J_\ascript\paren{\ket{\psi}\bra{\phi}}=K_\bscript\phi\otimes\psi,\quad J_\ascript ^*(\phi\otimes\psi )=\ket{\psi}\bra{K_\bscript\phi}
\end{equation*}
for all $\phi\in H_1$, $\psi\in H_2$.
\end{lem}
\begin{proof}
For $\phi\in H_1$, $\psi\in H_2$ we have that
\begin{align*}
J_\ascript\paren{\ket{\psi}\bra{\phi}}&=\sum _{i,j}\elbows{\psi _i,\ket{\psi}\bra{\phi}\phi _j}\phi _j\otimes\psi _i\\
   &=\sum _{i,j}\elbows{\phi ,\phi _j}\elbows{\psi _i,\psi}\phi _j\otimes\psi _i\\
   &=\sum _j\elbows{\phi ,\phi _j}\phi _j\otimes\sum _i\elbows{\psi _i,\psi}\psi _i=K_\bscript\phi\otimes\psi
\end{align*}
Moreover, we have $\ket{\psi}\bra{\phi}=J_\ascript ^*(K_\bscript\phi\otimes\psi )$ and replacing $\phi$ by $K_\bscript\phi$ gives
\begin{equation*}
J_\ascript ^*(\phi\otimes\psi )=\ket{\psi}\bra{K_\bscript\phi}\qedhere
\end{equation*}
\end{proof}

It has been said that isomorphisms between $\lscript _{\rmhs}(H_1,H_2)$ and $H_1\otimes H_2$ are based on the relation
$\ket{\psi}\bra{\phi}\leftrightarrow\phi\otimes\psi$. This is not entirely correct because the relation is anti-linear in the second argument on the left. The correct relation is given in Lemma~\ref{lem15}.

We now show that $J_\ascript$ is canonical in the sense that it is essentially independent of the bases $\ascript$. To be precise, $J_\ascript$ is unique to within a conjugation.

\begin{cor}    
\label{cor16}
Let $\ascript '=\brac{\phi '_i,\psi '_j}$ be another set of bases and let $\bscript '=\brac{\phi '_i}$. Then
\begin{align*}
J_{\ascript '}(A)&=\sum\elbows{\psi _i,A\phi _j}(K_{\bscript '}\phi _j)\otimes\psi _i
\intertext{and}
J_{\ascript '}^*(\alpha )&=\sum\elbows{\phi _j\otimes\psi _i,\alpha}\ket{\psi _i}\bra{K_{\bscript '}\phi _j}
\end{align*}
\end{cor}
\begin{proof}
Applying Lemma~\ref{lem15} we have that
\begin{align*}
J_{\ascript '}(A)&=J_{\ascript '}\paren{\sum\elbows{\psi _i,A\phi _j}\ket{\psi _i}\bra{\phi _j}}
  =\sum\elbows{\psi _i,A\phi _j}J_{\ascript '}\paren{\ket{\psi _i}\bra{\phi _j}}\\
  &\sum\elbows{\psi _i,A\phi _j}(K_{\bscript '}\phi _j)\otimes\psi _i
\end{align*}
The second equation is similar.
\end{proof}

We have seen that $J_\ascript\colon\lscript _{\rmhs}(H_1,H_2)\to H_1\otimes H_2$ is a Hilbert space isomorphism. We now apply $J_\ascript$ 
to obtain a map
\begin{equation*}
R\colon\lscript\paren{\lscript _{\rmhs}(H_1,H_2)}\to\lscript (H_1\otimes H_2)
\end{equation*}
defined by $R(B)=J_\ascript BJ_\ascript ^*$. We also define the map
\begin{equation*}
S\colon\lscript (H_1\otimes H_2)\to\lscript\paren{\lscript _{\rmhs}(H_1,H_2)}
\end{equation*}
by $S(C)=J_\ascript ^*CJ_\ascript$. We treat $\lscript (H_1\otimes H_2)$ and $\lscript\paren{\lscript _{\rmhs}(H_1,H_2)}$ as $C^*$-algebras in the usual way.

\begin{thm}    
\label{thm17}
The maps, $R$ and $S$ are $C^*$-algebra isomorphisms and $S=R^{-1}$.
\end{thm}
\begin{proof}
Clearly, $R$ and $S$ are linear, $R(I)=I$ and
\begin{align*}
RS(C)&=J_\ascript S(C)J_\ascript ^*=J_\ascript J_\ascript ^*CJ_\ascript J_\ascript ^*=C\\
SR(B)&=J_\ascript ^*R(B)J_\ascript =J_\ascript ^*J_\ascript BJ_\ascript ^*J_\ascript =B
\end{align*}
It follows that $R$ and $S$ are bijections and $S=R^{-1}$. Since
\begin{equation*}
R(B_1B_2)=J_\ascript B_1B_2J_\ascript ^*=J_\ascript B_1J_\ascript ^*J_\ascript B_2J_\ascript ^*=R(B_1)R(B_2)
\end{equation*}
and similarly, $S(C_1C_2)=S(C_1)S(C_2)$ we conclude that $R$ and $S$ preserve products. Also
\begin{equation*}
R(B^*)=J_\ascript B^*J_\ascript ^*=(J_\ascript BJ_\ascript ^*)^*=R(B)^*
\end{equation*}
so $R$ and $S$ preserve adjoints. Since $J_\ascript$ and $J_\ascript ^*$ preserve norms and $J_\ascript ^*$ is surjective, we have
\begin{align*}
\doubleab{R(B)}&=\sup _{\doubleab{x}=1}\doubleab{R(B)x}=\sup _{\doubleab{x}=1}\doubleab{J_\ascript BJ_\ascript ^*x}
   =\sup _{\doubleab{x}=1}\doubleab{BJ_\ascript ^*x}\\
   &=\sup _{\doubleab{J_\ascript ^*x}=1}\doubleab{BJ_\ascript ^*x}=\sup _{\doubleab{y}=1}\doubleab{By}=\doubleab{B}
\end{align*}
Similarly, $\doubleab{S(C)}=\doubleab{C}$. Since $R$ and $S$ are bijections that preserve $C^*$-algebra operations, they are $C^*$-algebra isomorphisms.
\end{proof}

It follows from Theorem~\ref{thm17} that $R$ and $S$ are positive operators, that is, they send positive elements to positive elements. It is a little less clear that $R$ and $S$ are completely positive in the following sense.

\begin{lem}    
\label{lem18}
For all $A\in\lscript\paren{\lscript _{\rmhs}(H_1,H_2)}$, $B\in\lscript (H_1\otimes H_2)$ we have
\begin{equation*}
\sum _{i,j}B_i^*R(A_i^*A_j)B_j\ge 0
\end{equation*}
\end{lem}
\begin{proof}
This follows from
\begin{align*}
\sum _{i,j}B_i^*R(A_i^*A_j)B_j&=\sum _{i,j}B_i^*R(A_i)^*R(A_j)B_j=\sum _i\sqbrac{R(A_i)B_i}^*\sum _jR(A_j)B_j\\
   &=\sqbrac{\sum _iR(A_i)B_i}^*\sqbrac{\sum _jR(A_j)B_j}\ge 0\qedhere
\end{align*}
\end{proof}

The next result shows that $R$ preserves the HS operator structure.

\begin{thm}    
\label{thm19}
{\rm{(a)}}\enspace If $B\in\lscript _{\rmhs}\paren{\lscript _{\rmhs}(H_1,H_2)}$, then $R(B)\in\lscript _{\rmhs}(H_1\otimes H_2)$.
{\rm{(b)}}\enspace If $B_1,B_2\in\lscript _{\rmhs}\paren{\lscript _{\rmhs}(H_1,H_2)}$, then
$\elbows{R(B_1),R(B_2)}_{\rmhs}=\elbows{B_1,B_2}_{\rmhs}$.
\end{thm}
\begin{proof}
(a)\enspace If $B\in\lscript _{\rmhs}\paren{\lscript _{\rmhs}(H_1,H_2)}$ and $\brac{\phi _i}$, $\brac{\psi _i}$ are orthonormal bases for $H_1$, $H_2$, respectively, we have by Lemma~\ref{lem15} that
\begin{align*}
\rmtr\sqbrac{R(B)^*R(B)}&=\rmtr R(B^*B)=\rmtr (J_\ascript B^*BJ_\ascript ^*)\\
   &=\sum\elbows{\phi _i\otimes\psi _jJ_\ascript B^*BJ_\ascript ^*\phi _i\otimes\psi _j}\\
   &=\sum\elbows{J_\ascript ^*\phi _i\otimes\psi _j,B^*BJ_\ascript ^*\phi _i\otimes\psi _j}_{\rmhs}\\
   &=\sum\elbows{\ket{\psi _j}\bra{\phi _i},B^*B\ket{\psi _j}\bra{\phi _i}}_{\rmhs}\\
   &=\rmtr (B^*B)<\infty
\end{align*}
Hence, $R(B)\in\lscript _{\rmhs}(H_1\otimes H_2)$.
(b)\enspace If $B_1,B_2\in\lscript _{\rmhs}\paren{\lscript _{\rmhs}(H_1,H_2)}$ we have
\begin{align*}
\elbows{R(B_1),R(B_2)}_{\rmhs}&=\sum\elbows{J_\ascript B_1J_\ascript ^*\phi _i\otimes\psi _j,J_\ascript B_2J_\ascript ^*\phi _i\otimes\psi _j}\\
   &=\sum\elbows{J_\ascript B_1\ket{\psi _j}\bra{\phi _i},J_\ascript B_2\ket{\psi _j}\bra{\phi _i}}\\
   &=\sum\elbows{B_1\ket{\psi _j}\bra{\phi _i},B_2\ket{\psi _j}\bra{\phi _i}}\\
   &=\sum\elbows{\ket{\psi _j}\bra{\phi _i},B_1^*B_2\ket{\psi _j}\bra{\phi _i}}\\
   &=\rmtr (B_1^*B_2)=\elbows{B_1,B_2}_{\rmhs}\qedhere
\end{align*}
\end{proof}

An analogous result holds for $S\colon\lscript _{\rmhs}(H_1\otimes H_2)\to\lscript _{\rmhs}\paren{\lscript _{\rmhs}(H_1,H_2)}$.

\section{Comparison with the Choi Isomorphism}  
We now give alternative forms for the operator isomorphism considered in Section~1. For an orthonormal basis $\bscript =\brac{\phi _i}$ of
$H_1$, define the map $T_\bscript\colon\lscript _{\rmhs}(H_1,H_2)\to H_1\otimes H_2$ by $T_\bscript (A)=\sum\phi _j\otimes A\phi _j$. It is not
yet clear that $T_\bscript (A)\in H_1\otimes H_2$, but this will follow from our next result. In fact, we shall show that $T_\bscript =J_\ascript$. This alternative form $T_\bscript$ is simpler than $J_\ascript$ and it only relies on the basis $\bscript$ of $H_1$. This form will be useful in our comparison with Choi's isomorphism \cite{cho75,hz12} considered later. Unfortunately, $T_\bscript ^*$ in this new form is not as simple as
$J_\ascript ^*$.

\begin{lem}    
\label{lem21}
$T_\bscript =J_\ascript$
\end{lem}
\begin{proof}
For any $A\in\lscript _{\rmhs}(H_1,H_2)$ and orthonormal basis $\brac{\psi _i}$ of $H_2$, letting $\ascript =\brac{\phi _i,\psi _j}$ we have
$A\phi _j=\sum\limits _i\elbows{\psi _i,A\phi _j}\psi _i$. Hence,
\begin{align*}
T_\bscript (A)&=\sum _j\phi _j\otimes A\phi _j=\sum _j\phi _j\otimes\sum\elbows{\psi _i,A\phi _j}\psi _i\\
   &=\sum _{i,j}\elbows{\psi _i,A\phi _j}\phi _j\otimes\psi _i=J_\ascript (A)\qedhere
\end{align*}
\end{proof}

Corresponding to the bases $\ascript =\brac{\phi _i,\psi _j}$ define the operators $P_i\colon H_1\otimes H_2\to H_2$ by
\begin{equation*}
P_i(\alpha )=\sum _j\elbows{\phi _i\otimes\psi _j,\alpha}\psi _j
\end{equation*}

\begin{lem}    
\label{lem22}
{\rm{(a)}}\enspace $P_i(\phi\otimes\psi )=\elbows{\phi _i,\phi}\psi$ for all $\phi\in H_1$, $\psi\in H_2$.
{\rm{(b)}}\enspace $\sum\phi _i\otimes P_i\alpha =\alpha$ for all $\alpha\in H_1\otimes H_2$.
{\rm{(b)}}\enspace $P_i^*(\beta )=\sum _j\elbows{\psi _j,\beta}\phi _i\otimes\psi _j$ for all $\beta\in H_2$.
\end{lem}
\begin{proof}
(a)\enspace For all $\phi\in H_1$, $\psi\in H_2$ we have
\begin{equation*}
P_i(\phi\otimes\psi )=\sum _j\elbows{\phi _i\otimes\psi _j,\phi\otimes\psi}\psi _j=\sum _j\elbows{\phi _i,\phi}\elbows{\psi _j,\psi}\psi _j
   =\elbows{\phi _i,\phi}\psi
\end{equation*}
(b)\enspace For all $\alpha\in H_1\otimes H_2$ we have
\begin{equation*}
\sum\phi _i\otimes P_i\alpha =\sum _i\phi _i\otimes\sum _j\elbows{\phi _i\otimes\psi _j,\alpha}\psi _j
   =\sum _{i,j}\elbows{\phi _i\otimes\psi _j,\alpha}\phi _i\otimes\psi _j=\alpha
\end{equation*}
(c)\enspace For all $\phi\in H_1$, $\psi$, $\beta\in H_2$, applying (a) gives
\begin{align*}
\elbows{\phi\otimes\psi ,P_i^*(\beta )}&=\elbows{P_i(\phi\otimes\psi ),\beta}=\elbows{\elbows{\phi _i,\phi}\psi ,\beta}
   =\elbows{\phi ,\phi _i}\elbows{\psi ,\beta}\\
   &=\elbows{\phi ,\phi _i}\sum _j\elbows{\psi ,\psi _j}\elbows{\psi _j,\beta}
     =\sum _j\elbows{\psi _j,\beta}\elbows{\phi ,\phi _i}\elbows{\psi ,\psi _j}\\
     &=\elbows{\phi\otimes\psi ,\sum _j\elbows{\psi _j,\beta}\phi _i\otimes\psi _j}
\end{align*}
The result follows.
\end{proof}

We see from Lemma~\ref{lem22}(a) that $P_i$ is analogous to a partial trace in the direction $\phi _i$.

\begin{lem}    
\label{lem23}
For all $\alpha\in H_1\otimes H_2$ we have $T_\bscript ^*(\alpha )=\sum\limits _i\ket{P_j\alpha}\bra{\phi _j}$.
\end{lem}
\begin{proof}
For $\alpha\in H_1\otimes H_2$ we obtain
\begin{align*}
\sum _j\ket{P_j\alpha}\bra{\phi _j}&=\sum _j\ket{\sum _i\elbows{\phi _j\otimes\psi _i,\alpha}\psi _i}\bra{\phi _j}\\
&=\sum _{i,j}\elbows{\phi _j\otimes\psi _i,\alpha}\ket{\psi _i}\bra{\phi _j}=J_\ascript ^*(\alpha )=T_\bscript ^*(\alpha )\qedhere
\end{align*}
\end{proof}

We now consider the Choi isomorphism \cite{cho75,hz12}. For simplicity, we assume that $H_1=H_2=H$. We also have to assume that
$\dim H<\infty$. Then $\lscript _{\rmhs}(H)=\lscript (H)$ and $T_\bscript\colon\lscript (H)\to H\otimes H$ is defined as before, where
$\bscript =\brac{\phi _1,\phi _2,\ldots ,\phi _n}$ is an orthonormal basis for $H$. Define the vector $\phi _+\in H\otimes H$ by
$\phi _+=\sum\limits _{j=1}^n\phi _j\otimes\phi _j$. This is the point at which we need $\dim H<\infty$. We then write
\begin{align}                
\label{eq21}
T_\bscript (A)&=\sum\phi _j\otimes A\phi _j=\sum (I\otimes A)\phi _j\otimes\phi _j=(I\otimes A)\sum\phi _j\otimes\phi _j\notag\\
   &=(I\otimes A)\phi _+
\end{align}
We extend \eqref{eq21} to a map $C_\bscript\colon\lscript\paren{\lscript (H)}\to\lscript (H\otimes H)$ by defining
\begin{equation}                
\label{eq22}
C_\bscript (B)=(I\otimes B)\ket{\phi _+}\bra{\phi _+}
\end{equation}
If we normalize $\phi _+$ (which is usually done), we can think of \eqref{eq22} as extending \eqref{eq21} from a vector state $\phi _+$ to the corresponding pure state $\ket{\phi _+}\bra{\phi _+}$. To see that $C_\bscript (B)\in\lscript (H\otimes H)$ we have that 
\begin{align*}
C_\bscript (B)&=(I\otimes B)\ket{\sum\phi _i\otimes\phi _i}\bra{\sum\phi _j\otimes\phi _j}
   =(I\otimes B)\sum _{i,j}\ket{\phi _i\otimes\phi _i}\bra{\phi _j\otimes\phi _j}\\
   &=(I\otimes B)\sum _{i,j}\ket{\phi _i}\bra{\phi _j}\otimes\ket{\phi _i}\bra{\phi _j}
   =\sum _{i,j}\paren{\ket{\phi _i}\bra{\phi _j}\otimes B\ket{\phi _i}\bra{\phi _j}}
\end{align*}
It is shown in \cite{cho75,hz12} that $C_\bscript$ is a Hilbert space isomorphism; that is, $C_\bscript$ is a linear bijection that preserves the HS inner product.

Notice that for $B_1,B_2\in\lscript\paren{\lscript (H)}$ we have
\begin{equation*}
C_\bscript (B_1B_2)=(I\otimes B_1B_2)\ket{\phi _+}\bra{\phi _+}
\end{equation*}
which, in general does not equal
\begin{equation*}
C_\bscript (B_1)C_\bscript (B_2)=(I\otimes B_1)\ket{\phi _+}\bra{\phi _+}(I\otimes B_2)\ket{\phi _+}\bra{\phi _+}
\end{equation*}
so $C_\bscript$ does not preserve products. In fact,
\begin{equation*}
C_\bscript (I)=\ket{\phi _+}\bra{\phi _+}\ne I
\end{equation*}
so we have
\begin{equation*}
C_\bscript (IB_2)=C_\bscript (B_2)\ne C_\bscript (I_1)C_\bscript (B_2)=\ket{\phi _+}\bra{\phi _+}(I\otimes B_2)\ket{\phi _+}\bra{\phi _+}
\end{equation*}
in general. Since $R_\ascript (I)=I$, this shows that $C_\bscript\ne R_\ascript$ so the two isomorphisms are different. The main advantage of
$C_\bscript$ over $R_\ascript$ is that $C_\bscript$ takes completely positive elements of $\lscript\paren{\lscript (H)}$ to positive operators in
$\lscript (H\otimes H)$.

\section{Discussion}  
Why is the isomorphism $R\colon\lscript\paren{\lscript _{\rmhs}(H_1,H_2)}\to\lscript (H_1\otimes H_2)$ important? One reason is that in the particular case when $H_1=H_2=H$, then the positive elements $\rho\in\lscript _{\rmhs}(H)$ with $\rmtr (\rho )=1$ represent quantum states. In particular, dynamics are given by one-parameter groups $U_t\in\lscript\paren{\lscript _{\rmhs}(H)}$. Also, quantum operations and quantum channels \cite{hz12,nc00} are represented by completely positive elements in $\lscript\paren{\lscript _{\rmhs}(H_1,H_2)}$. For such a map
$B$, $R(B)$ becomes an operator in $\lscript (H_1\otimes H_2)$. Now we can employ the well-developed theory of bounded operators on a Hilbert space to study properties of $B$. In particular, the concatenation of two channels can become quite complicated but since
$R(B_1B_2)=R(B_1)R(B_2)$ this becomes a simple product of operators. For example a quantum channel
$B\in\lscript\paren{\lscript _{\rmhs}(H)}$ has a \textit{Kraus decomposition} $B(A)=\sum M_i^*AM_i$, $A\in\lscript _{\rmhs}(H)$, $M_i\in\lscript (H)$ with $\sum M_iM_i^*=I$. If we concatenate $B$ with a channel $C$ where $C(A)=\sum N_i^*AN_i$, then
\begin{equation*}
CB(A)=\sum _{i,j}N_j^*M_i^*AM_iN_j
\end{equation*}
which can be complicated to analyze.

In more detail, if $B(A)=\sum M_i^*AM_i$ and $\ascript =\brac{\phi _i}$ is an orthonormal basis for $H$, then for all $\alpha\in H\otimes H$ we that
\begin{align}                
\label{eq31}
R(B)\alpha&=J_\ascript BJ_\ascript ^*\alpha =J_\ascript\sum _iM_i^*J_\ascript ^*(\alpha )M_i\notag\\
  &=\sum _{r,s}\elbows{\phi _r,\sum _iM_i^*J_\ascript ^*(\alpha )M_i\phi _s}\phi _s\otimes\phi _r\notag\\
  &=\sum _s\phi _s\otimes\sum _iM_i^*J_\ascript ^*(\alpha )M_i\phi _s
\end{align}
Letting $M_\alpha\in\lscript (H)$ be the operator
\begin{align*}
M_\alpha&=\sum _iM_i^*J_\ascript ^*(\alpha )M_i=\sum _iM_i^*\sum _{r,s}\elbows{\phi _r\otimes\phi _s,\alpha}\ket{\phi _s}\bra{\phi _r}M\\
   &=\sum _{r,s}\elbows{\phi _r\otimes\phi _s,\alpha}\sum _i\ket{M_i^*\phi _s}\bra{M_i^*\phi _r}
\end{align*}
Equation \eqref{eq31} becomes
\begin{equation}                
\label{eq32}
R(B)\alpha =\sum _s\phi _s\otimes M_\alpha\phi _s
\end{equation}
If $\dim H<\infty$, and $\phi _+=\sum _s\phi _s\otimes\phi _s$, we can write \eqref{eq32} as
\begin{equation*}
R(B)\alpha =\sum _s(I\otimes M_\alpha )\phi _s\otimes\phi _s=(I\otimes M_\alpha )\phi _+
\end{equation*}
which is similar to \eqref{eq21}.

We can turn things around and consider the isomorphism $S\colon\lscript (H_1\otimes H_2)\to\lscript\paren{\lscript _{\rmhs}(H_1,H_2)}$. In this way we can study entanglement from a different point of view. For example, if $\alpha\in H_1\otimes H_2$ is a factorized (or product) vector state, then $\alpha =\phi\otimes\psi$, $\phi\in H_1$, $\psi\in H_2$ where $\doubleab{\phi}=\doubleab{\psi}=1$. It is convenient to form bases
$\ascript =\brac{\phi _i,\psi _i}$ with $\phi _1=\phi$, $\psi _1=\psi$ so that $J_\ascript ^*(\alpha )=\ket{\psi}\bra{\phi}$. If
$\dim H_1=\dim H_2=n<\infty$, then any vector state $\alpha$ has a \textit{Schmidt decomposition} $\alpha =\sum\lambda _i\phi _i\otimes\psi _i$, where $\lambda _i\ge 0$ and $\phi _i$, $\psi _i$ are orthonormal basis for $H_1$, $H_2$, respectively \cite{hz12,nc00}. We then have that
$\alpha$ is entangled if and only if at least two of the $\lambda _i$ are positive. Again, letting $\ascript =\brac{\phi _i,\psi _j}$ we have that
\begin{equation*}
J_\ascript ^*(\alpha )=J_\ascript ^*\paren{\sum\lambda _i\phi _i\otimes\psi _i}=\sum\lambda _i\ket{\psi _i}\bra{\phi _i}
\end{equation*}
Corresponding to the factorized vector state $\alpha =\phi\otimes\psi$ we have the factorized pure state
\begin{equation*}
P_\alpha\ket{\alpha}\bra{\alpha}=\ket{\phi\otimes\psi}\bra{\phi\otimes\psi}=\ket{\phi}\bra{\phi}\otimes\ket{\psi}\bra{\psi}=P_\phi\otimes P_\psi
\end{equation*}
It is easy to check that
\begin{equation*}
S(P_\alpha )=S(P_\phi\otimes P_\psi )=P_{\ket{\psi}\bra{\phi}}
\end{equation*}
More generally, for a factorized operator $A\otimes B\in\lscript (H_1\otimes H_2)$, $A\in\lscript (H_1)$, $B\in\lscript (H_2)$ we have that
\begin{align*}
S(A\otimes B)\ket{\psi}\bra{\phi}&=J_\ascript ^*(A\otimes B)J\ket{\psi}\bra{\phi}=J_\ascript ^*(A\otimes B)\phi\otimes\psi\\
   &=J_\ascript ^*(A\phi\otimes B\psi )=\ket{B\psi}\bra{A\phi}
   \end{align*}
From this we can compute $S(A\otimes B)C$ for any $C\in\lscript _{\rmhs}(H_1,H_2)$ because
\begin{equation*}
C=\sum\elbows{\ket{\psi _i}\bra{\phi _j},C\ket{\psi _i}\bra{\phi _j}}_{\rmhs}\ket{\psi _i}\bra{\phi _j}
\end{equation*}
We leave a deeper analysis for later studies


\begin{thebibliography}{99}
\bibitem{cho75}M.~Choi, Completely positive maps on complex matrices, \textit{Linear.~Alg.~Appl.} \textbf{10},285--290 (1975).
\bibitem{hz12}T.~Heinosaari and M.~Ziman, The Mathematical Language of Quantum Theory, Cambridge University Press, Cambridge, 2012.
\bibitem{nc00}M.~Nielson and I.~Chuang, Quantum Computation and Quantum Information, Cambridge University Press, Cambridge, 2000.

\end{thebibliography}
\end{document}